\documentclass[11pt]{article}

\usepackage[utf8]{inputenc} 
\usepackage{amsmath}
\usepackage{graphicx}
\usepackage{subcaption}
\usepackage[top=28.5mm, bottom=28.5mm, left=27mm, right=27mm]{geometry}
\usepackage{amsfonts}
\usepackage{mathtools}
\usepackage{subcaption}
\usepackage{amssymb}
\usepackage{amsthm}
\usepackage{chngpage}
\graphicspath{ {Images/} }
\usepackage{hyperref}
\hypersetup{
	colorlinks=true,
	linkcolor=black,
	citecolor=black,
	filecolor=black,      
	urlcolor=black,
}
\usepackage{enumerate}
\usepackage{makeidx}
\usepackage{thmtools}
\usepackage{thm-restate}
\usepackage{tikz}
\usetikzlibrary{positioning,chains,fit,shapes,calc,patterns}
\usepackage{color}
\usepackage{comment}

\makeindex

\makeatletter
\renewenvironment{proof}[1][\proofname] {\par\pushQED{\qed}\normalfont\topsep6\p@\@plus6\p@\relax\trivlist\item[\hskip\labelsep\bfseries#1\@addpunct{.}]\ignorespaces}{\popQED\endtrivlist\@endpefalse}
\makeatother

\newtheorem{proposition}{Proposition}[section]
\newtheorem{conjecture}[proposition]{Conjecture}
\newtheorem{lemma}[proposition]{Lemma}
\newtheorem{theorem}[proposition]{Theorem}

\theoremstyle{definition}
\newtheorem{definition}[proposition]{Definition}

\newtheorem{remark}[proposition]{Remark}
\newtheorem{problem}[proposition]{Problem}

\newtheorem*{remark*}{Remark}
\newtheorem*{theorem*}{Theorem}
\newtheorem{claim}[proposition]{Claim}

\usepackage[shortlabels]{enumitem}
\usepackage[capitalise]{cleveref}
\usepackage{url}

\usepackage[numbers,sort]{natbib}

\renewcommand{\epsilon}{\varepsilon}
\title{Ordered Ramsey numbers of powers of paths}
\author{Ant\'onio Gir\~ao\thanks{Mathematical Institute, University of Oxford, United Kingdom. Research supported by EPSRC grant EP/V007327/1.
Email: \textbf{girao@maths.ox.ac.uk}}
\and
Barnabás Janzer\thanks{Mathematical Institute, University of Oxford, United Kingdom. Research supported by a fellowship at Magdalen College.
Email: \textbf{barnabas.janzer@magd.ox.ac.uk}}
\and
Oliver Janzer\thanks{Department of Pure Mathematics and Mathematical Statistics, University of Cambridge, United Kingdom. Research supported by a fellowship at Trinity College. Email: \textbf{oj224@cam.ac.uk}.}
}

\date{\vspace{-21pt}}
\begin{document}
	\maketitle
	
	\begin{abstract}
	Given two vertex-ordered graphs $G$ and $H$, the ordered Ramsey number $R_<(G,H)$ is the smallest $N$ such that whenever the edges of a vertex-ordered complete graph $K_N$ are red/blue-coloured, then there is a red (ordered) copy of $G$ or a blue (ordered) copy of $H$. Let $P_n^t$ denote the $t$-th power of a monotone path on $n$ vertices. The ordered Ramsey numbers of powers of paths have been extensively studied. We prove that there exists an absolute constant $C$ such that $R_<(K_s,P_n^t)\leq R(K_s,K_t)^{C} \cdot n$ holds for all $s,t,n$, which is tight up to the value of $C$. As a corollary, we obtain that there is an absolute constant $C$ such that $R_<(K_n,P_n^t)\leq n^{Ct}$. These results resolve a problem and a conjecture of Gishboliner, Jin and Sudakov. Furthermore, we show that $R_<(P_n^t,P_n^t)\leq n^{4+o(1)}$ for any fixed $t$. This answers questions of Balko, Cibulka, Král and  Kyn{\v{c}}l, and of Gishboliner, Jin and Sudakov.
	\end{abstract}
	
\section{Introduction}

Given two graphs $G$ and $H$ such that their vertex sets $V(G),V(H)$ are both (totally) ordered, we say that $G$ contains a copy of $H$ if $G$ has a subgraph $H'$ isomorphic to $H$ such that the vertices of $H$ and $H'$ are ordered in the same way -- that is, if $f$ denotes the corresponding graph isomorphism from $H$ to $H'$, then $v<w$ implies $f(v)<f(w)$ for all $v,w\in V(H)$. Given two vertex-ordered graphs $G,H$, the ordered Ramsey number $R_<(G,H)$ is defined to be the smallest positive integer $N$ such that whenever the edges of a vertex-ordered complete graph on $N$ vertices are red/blue-coloured, we can find a red copy of $G$ or a blue copy of $H$. Note that in the case of complete graphs, $R_<(K_r,K_s)$ coincides with the classical Ramsey number $R(K_r,K_s)=R(r,s)$, irrespective of orderings. For general (ordered) graphs, we clearly have $R(G,H)\leq R_<(G,H)\leq R(K_{|G|},K_{|H|})$.

The systematic study of ordered Ramsey numbers of graphs (and hypergraphs) was initiated by Conlon, Fox, Lee and Sudakov~\cite{conlon2017ordered} and, independently, by Balko, Cibulka, Král and  Kyn{\v{c}}l~\cite{balko2015ramsey}. However, the study of ordered (hyper)graph Ramsey numbers for paths has a much longer history, dating back to the origins of Ramsey theory. For example, the proof of the famous Erd\H{o}s--Szekeres theorem~\cite{erdos1935combinatorial} about monotone subsequences gives $R_<(P_s,P_n)=R_<(K_s,P_n)=(s-1)(n-1)+1$, where $P_n$ denotes the path on $n$ vertices (ordered monotonically), and the `cups-caps' theorem in the same paper corresponds to a hypergraph version of this problem. The Ramsey numbers of paths have important connections to other areas, in particular to convex geometry. Erd\H{o}s and Szekeres used their `cups-caps' theorem to prove that every set of $\binom{2n-4}{n-2}+1$ points in the plane in general position contains $n$ points in convex position. Fox, Pach, Sudakov and Suk~\cite{fox2012erdHos} also studied the hypergraph Ramsey problem for tight paths, and gave further geometric applications. Their results were later extended by Moshkovitz and Shapira~\cite{moshkovitz2014ramsey}. Furthermore, Mubayi and Suk~\cite{mubayi2017off} established a connection between ordered Ramsey numbers of tight paths and multicolour Ramsey numbers of cliques.

In this paper we focus on bounding $R_<(G,H)$ where $G$ is a power of a path, and $H$ is either a clique or another path power. Given positive integers $t$ and $n$, the $t$-th power of a path on $n$ vertices, $P_n^t$, is the graph with vertex set $[n]=\{1,\dots,n\}$, where $i$ and $j$ are adjacent if and only if $0<|i-j|\leq t$. (The vertex set of $P_n^t$ will be ordered in the obvious way.) The problem of bounding the Ramsey number of $P_n^t$ is important not only because it gives a natural extension of the celebrated Erd\H os--Szekeres Theorem, but also since every $n$-vertex ordered graph with bandwidth $t$ is a subgraph of $P_n^t$. Answering a question of Conlon, Fox, Lee and Sudakov~\cite{conlon2017ordered}, Balko, Cibulka, Král and  Kyn{\v{c}}l~\cite{balko2015ramsey} proved an upper bound $R_<(P_n^t,P_n^t)=O_t(n^{129t})$. The authors of~\cite{balko2015ramsey} asked to determine the growth rate of $R_<(P_n^t,P_n^t)$. In the case $t=2$, Mubayi~\cite{mubayi2017variants} improved their bound to $R_<(P_n^2,P_n^2)=O(n^{19.5})$, and very recently, Gishboliner, Jin and Sudakov~\cite{gishboliner2023ramsey} proved that $R_<(P_n^t,P_n^t)=O_t(n^{4t-2})$ for all $t$. The authors of~\cite{gishboliner2023ramsey} asked whether in fact $R_<(P_n^t,P_n^t)=O_t(n^{C})$ holds for some constant $C$, and commented that even improving the exponent to $o(t)$ would be interesting. Our first result is a positive answer to their question. In fact, the constant exponent we obtain improves the best known bound in all cases $t>1$.

\begin{theorem}\label{theorem_pathvspath}
	Let $t$ be a positive integer and $\epsilon>0$. Then there is some $C=C(\epsilon,t)>0$ such that, for all positive integers $n$, we have $$R_<(P_n^t,P_n^t)\leq Cn^{4+\epsilon}.$$
\end{theorem}

For the problem of path powers versus cliques, Gishboliner, Jin and Sudakov~\cite{gishboliner2023ramsey} proved that $R_<(P_n^t,K_n)=O_t(n^{t(2t-1)})$ -- this improved earlier bounds of Mubayi and Suk~\cite{mubayi2023ramsey} and of Conlon, Fox, Lee and Sudakov~\cite{conlon2017ordered}. Gishboliner, Jin and Sudakov~\cite{gishboliner2023ramsey} conjectured that the exponent can even be chosen to be linear in $t$. The second main result in this paper verifies this conjecture.

\begin{theorem}\label{theorem_pathvsclique}
	There exists $C>0$ such that, for all positive integers $n>t$,
	$$R_<(P_n^t,K_n)\leq n^{Ct}.$$
\end{theorem}

\noindent Theorem \ref{theorem_pathvsclique} is tight up to the value of $C$ since $R_<(P_n^t,K_n)\geq R(K_{t+1},K_n)\geq \tilde{\Omega}_t(n^{\frac{t+2}{2}})$ (see \cite{spencer1977asymptotic}).

In fact, we derive Theorem \ref{theorem_pathvsclique} from an essentially tight bound for $R_<(P_n^t,K_s)$ which is valid for all $s$, $t$ and $n$.
Mubayi and Suk \cite{mubayi2023ramsey} showed that $R_<(P_n^t,K_s)=O_{s,t}(n\log^{s-2} n)$ and conjectured that $R_<(P_n^t,K_s)=O_{s,t}(n)$. Gishboliner, Jin and Sudakov~\cite{gishboliner2023ramsey} proved this conjecture by showing that $R_<(P_n^t,K_s)\leq (24s^3)^{st}n$. They asked to determine the correct dependence of the exponent on $s$ and $t$. We resolve this problem by proving the following bound.

\begin{theorem} \label{thm:pathvscliquegeneral}
    There exists an absolute constant $C$ such that for all $s,t\geq 2$ and $n>t$, we have
    $$R_<(P_n^t,K_s)\leq R(K_s,K_t)^C \cdot n.$$
\end{theorem}

\noindent This is tight up to the value of $C$, since (as was observed in \cite{gishboliner2023ramsey}), $R_<(P_n^t,K_s)>(R(K_{s+1},K_{t+1})-1)\cdot (n-1)/t$. While our proof of Theorem \ref{thm:pathvscliquegeneral} builds on ideas of Gishboliner, Jin and Sudakov, it contains several new ideas and is considerably shorter than their proof of $R_<(P_n^t,K_s)\leq (24s^3)^{st}n$.

\medskip

The rest of this paper is organized as follows. In Section \ref{sec:pathvspath} we prove Theorem \ref{theorem_pathvspath}. In Section~\ref{sec:pathvscliquegeneral} we prove Theorem \ref{thm:pathvscliquegeneral}, from which Theorem \ref{theorem_pathvsclique} follows immediately by setting $s=n$. In Section~\ref{sec:remarks} we give some concluding remarks and state some open problems arising from our work.

\paragraph{Notation.} We write $|G|$ for the number of vertices of the graph $G$.
	
\section{Path power versus path power}	\label{sec:pathvspath}
	In this section, we prove Theorem~\ref{theorem_pathvspath}. We will use an inductive argument, during which we will find the following, slightly larger ordered graphs. Given positive integers $m$ and $s$, let $Q_m^{t,s}$ denote the graph obtained by taking the monotonically ordered path on $m$ vertices, replacing the endpoints of the path by cliques of size $s$, and replacing the middle $(m-2)$ vertices by cliques of size $t$ (such that the edges between `adjacent' cliques form complete bipartite graphs). The vertices of $Q_m^{t,s}$ are ordered according to which clique they belong to (and arbitrarily inside each clique). Thus, if $s>t$, then $Q_m^{t,s}$ contains $P_m^t$, so Theorem \ref{theorem_pathvspath} follows immediately from the following result.
	
	\begin{lemma}\label{lemma_Q}
		Let $t$ be a positive integer and let $\epsilon>0$. Then there exist some $s=s(\epsilon,t)>t$ and $C=C(\epsilon,t)>0$ such that, for all positive integers $\ell$ and $n$, we have
		$$R_<(Q_\ell^{t,s},Q_n^{t,s})\leq C(\ell n)^{2+\epsilon}.$$
	\end{lemma}
	\begin{proof}
		Let $s=s(\epsilon,t)$ be sufficiently large (to be specified later). We prove the statement by induction on $\ell+n$. Let $\ell_0=\ell_0(\epsilon)$ be a sufficiently large positive integer (to be specified later). When $\ell\leq \ell_0$, then $R_<(Q_\ell^{t,s},Q_n^{t,s})\leq R_<(K_{\ell_0 s},Q_n^{t,s})\leq R_<(K_{\ell_0 s},P_{sn}^{2s})$. It was shown by Gishboliner, Jin and Sudakov~\cite{gishboliner2023ramsey} that for any positive integers $p, q$, there is some $D=D(p,q)$ such that we have $R_<(K_{p},P_m^q)\leq Dm$ for all $m$. It follows that if $\ell\leq \ell_0$ then $R_<(Q_\ell^{t,s},Q_n^{t,s})\leq D'n$ for some $D'=D'(\epsilon,t)$. (We remark that we did not need to use the bound proved in \cite{gishboliner2023ramsey} -- the weaker bound $R_<(K_p,P_m^q)=O_{p,q}(m\log^{p-2} m)$ from \cite{mubayi2023ramsey}, which has a short proof, would also suffice.) Thus, the result follows if $\ell\leq \ell_0$ (or $n\leq \ell_0$) provided that $C$ is chosen to be sufficiently large.
		
		Now assume that $\ell,n>\ell_0$. Consider a red/blue-coloured complete graph on vertex set $[N]$, where $N\geq C(\ell n)^{2+\epsilon}$. For convenience, we will assume that $N$ is even -- this can be achieved by deleting an arbitrary vertex if $N$ is odd, so we will instead have $N\geq C(\ell n)^{2+\epsilon}-1$. Without loss of generality, we may assume that there is a set $V_1\subseteq[N/2]$ just that $|V_1|\geq N/4$ and each vertex in $V_1$ sends at least $N/4$ red edges to $[N/2+1,N]$. By induction, we have $R_<(Q_{\lceil\ell/2\rceil}^{t,s},Q_n^{t,s})\leq C(\lceil\ell/2\rceil n)^{2+\epsilon}$, which is less than $N/4\geq\frac{1}{4}C(\ell n)^{2+\epsilon}-1$ if $\ell\geq \ell_0$ and $\ell_0$ is sufficiently large (and $C>1$). Thus, $G[V_1]$ contains a red $Q_{\lceil \ell/2\rceil}^{t,s}$ or a blue $Q_{n}^{t,s}$. In the latter case we are done, so let us assume that we have a red copy of $Q_{\lceil \ell/2\rceil}^{t,s}$ in $G[V_1]$. Let $X$ denote its vertex set and let $S$ denote the last $s$ elements of $X$.
		
		Let $\lambda=\frac{2^{\epsilon/2}-1}{2\cdot 2^{\epsilon/2}-1}>0$, and let $V_2$ be the set of vertices in $[N/2+1,N]$ which have at least $\lambda s$ red neighbours in $S$. Note that the total number of red edges between $S$ and $[N/2+1,N]$ is at least $sN/4$, and hence
		$$sN/4\leq |V_2|s+(N/2-|V_2|)\lambda s.$$
		It follows that
		\begin{align*}
			|V_2|&\geq \frac{1/4-\lambda/2}{1-\lambda}N=\frac{1}{2^{2+\epsilon/2}}N\geq \frac{1}{2^{2+\epsilon/2}}\left(C(\ell n)^{2+\epsilon}-1\right).
		\end{align*}
		Hence, by the induction hypothesis, $|V_2|> R_<(Q_{\lfloor \ell/2\rfloor}^{t,s},Q_n^{t,s})$ (if $\ell_0$ is large enough). It follows that in $G[V_2]$ we can find a red $Q_{\lfloor \ell/2\rfloor}^{t,s}$ or a blue $Q_n^{t,s}$. In the latter case we are done, so we may assume that we have a red copy of $Q_{\lfloor \ell/2\rfloor}^{t,s}$. Let us write $Y$ for its vertex set and $S'$ for the first $s$ elements of $Y$.
		
		Note that, by the definition of $V_2$, the number of red edges between $S$ and $S'$ is at least $\lambda s^2$. By the well-known result of Kővári, Sós and Turán~\cite{kHovari1954problem}, if $s$ is sufficiently large in terms of $\epsilon$ and $t$, then there is a red $K_{t,t}$ in the bipartite graph induced by vertex classes $S$ and $S'$. Writing $U\subseteq S$ and $U'\subseteq S'$ for the vertices of this $K_{t,t}$, it follows that there is a red copy of $Q_\ell^{t,s}$ on vertex set $(X\setminus S)\cup U\cup U'\cup (Y\setminus S')$, finishing the proof of the lemma.
	\end{proof}
\begin{proof}[Proof of Theorem~\ref{theorem_pathvspath}]
	The result follows immediately from Lemma~\ref{lemma_Q} since $Q_n^{t,s}$ contains $P_n^t$ provided that $s\geq t$.
\end{proof}

\section{Path power versus clique} \label{sec:pathvscliquegeneral}

In this section, we prove Theorem \ref{thm:pathvscliquegeneral}. As we have already mentioned in the introduction, certain parts of our proof are inspired by arguments of Gishboliner, Jin and Sudakov \cite{gishboliner2023ramsey}. For the reader's benefit we first prove Theorem \ref{thm:pathvscliquegeneral} subject to two lemmas whose proofs are postponed to subsequent sections.

\begin{lemma} \label{lemma_findKt+tfroms}
    Let $s,t\geq 2$ and let $N\geq (\binom{s+t}{s})^{10}$. Let $V_1,\dots,V_s$ be pairwise disjoint sets of size $N$. Let $G$ be a graph with vertex set $V_1\cup \dots \cup V_s$ and no independent set of size $s$. Then there exist some $1\leq i< j\leq s$ and subsets $T_i\subset V_i$ and $T_j\subset V_j$ of size $t$ such that $T_i\cup T_j$ is a clique in $G$.
\end{lemma}

\noindent The proof of this lemma, which is fairly standard, will be given in Section \ref{sec:auxlemma}.
The next definition and lemma will play a key role in our proof.
Here and below we write $S<T$ for vertex sets $S,T$ in some ordered graph if $\max(S)<\min(T)$.

\begin{definition}
    Given a vertex-ordered graph $G$ and positive integers $s,t,k$, an $(s,t)$-\emph{chain} of length $k$ consists of sets $A_1<B_{1,1}<\dots <B_{1,p_1}<A_2<B_{2,1}<\dots<B_{2,p_2}<\dots<A_k<B_{k,1}<\dots<B_{k,p_k}\subset V(G)$ such that $p_1,\dots,p_{k-1}\geq 1$ but $p_k$ may be 0, and 
    \begin{itemize}
        \item $B_{i,j}$ is a clique for all $i\in [k], 1\leq j\leq p_i$,
        \item $A_i$ is complete to $B_{i,1}$ for all $i\in [k]$,
        \item $B_{i,j}$ is complete to $B_{i,j+1}$ for all $i\in [k],1\leq j\leq p_i-1$,
        \item $B_{i,p_i}$ is complete to $A_{i+1}$ for all $i\in [k-1]$,
        \item $|A_i|\geq \left(\binom{s+t}{s}\right)^{10}$ for all $i\in [k]$ and
        \item $|B_{i,j}|= t$ for all $i\in [k],1\leq j\leq p_i$.
    \end{itemize}
    See Figure \ref{figure_chain} for an illustration.
\end{definition}

\begin{figure}[h]
\begin{tikzpicture}[thin,
	every node/.style={draw,circle},
    every fit/.style={ellipse,draw,text width=1cm}	]
	
\begin{scope}[start chain=going right,node distance=1.5mm,inner sep=1.3pt]
	\foreach \i in {1,2,...,4}
	\node[%fsnode,
	on chain] (aa\i) {};
\end{scope}
\begin{scope}[xshift=1.2cm,yshift=0cm,start chain=going right,node distance=1mm, inner sep=0pt]
	\foreach \i in {1,2,...,3}
	\node[%fsnode,
	on chain,fill] () {};
\end{scope}
\begin{scope}[xshift=1.8cm,yshift=0cm,start chain=going right,node distance=1.5mm, inner sep=1.3pt]
	\foreach \i in {5,...,5}
	\node[%fsnode,
	on chain] (aa\i) {};
\end{scope}

	% the vertices of B_{1,1}
	\begin{scope}[xshift=3cm,yshift=0cm,start chain=going right,node distance=1.5mm, inner sep=1.3pt]
		\foreach \i in {1,2,3}
		\node[%fsnode,
		on chain] (ba\i) {};
	\end{scope}
	
	% the vertices of B_{1,2}
	\begin{scope}[xshift=4.8cm,yshift=0cm,start chain=going right,node distance=1.5mm,inner sep=1.3pt]
		\foreach \i in {1,2,3}
		\node[%ssnode,
		on chain] (bb\i) {};
	\end{scope}

	% A_2
\begin{scope}[xshift=6.6cm,yshift=0cm,start chain=going right,node distance=1.5mm,inner sep=1.3pt]
	\foreach \i in {1,2,...,4}
	\node[%fsnode,
	on chain] (ab\i) {};
\end{scope}
\begin{scope}[xshift=7.8cm,yshift=0cm,start chain=going right,node distance=1mm, inner sep=0pt]
	\foreach \i in {1,2,...,3}
	\node[%fsnode,
	on chain,fill] () {};
\end{scope}
\begin{scope}[xshift=8.4cm,yshift=0cm,start chain=going right,node distance=1.5mm, inner sep=1.3pt]
	\foreach \i in {5,...,5}
	\node[%fsnode,
	on chain] (ab\i) {};
\end{scope}

	% the vertices of B_{2,1}
\begin{scope}[xshift=9.6cm,yshift=0cm,start chain=going right,node distance=1.5mm, inner sep=1.3pt]
	\foreach \i in {1,2,3}
	\node[%fsnode,
	on chain] (ca\i) {};
\end{scope}
	
	% A_3
\begin{scope}[xshift=11.5cm,yshift=0cm,start chain=going right,node distance=1.5mm,inner sep=1.3pt]
	\foreach \i in {1,2,...,4}
	\node[%fsnode,
	on chain] (ac\i) {};
\end{scope}
\begin{scope}[xshift=12.7cm,yshift=0cm,start chain=going right,node distance=1mm, inner sep=0pt]
	\foreach \i in {1,2,...,3}
	\node[%fsnode,
	on chain,fill] () {};
\end{scope}
\begin{scope}[xshift=13.2cm,yshift=0cm,start chain=going right,node distance=1.5mm, inner sep=1.3pt]
	\foreach \i in {5,...,5}
	\node[%fsnode,
	on chain] (ac\i) {};
\end{scope}	
	
	% the vertices of B_{3,1}
\begin{scope}[xshift=14.4cm,yshift=0cm,start chain=going right,node distance=1.5mm, inner sep=1.3pt]
	\foreach \i in {1,2,3}
	\node[%fsnode,
	on chain] (da\i) {};
\end{scope}

	\node [fit=(aa1) (aa5),label=above:$A_{1}$,text width=1.45cm] {};		
	\node [fit=(ba1) (ba3),label=above:$B_{1,1}$,text width=0.5cm,pattern = north east lines] {};
	\node [fit=(bb1) (bb3),label=above:$B_{1,2}$,text width=0.5cm,pattern = north east lines] {};
	\node [fit=(ab1) (ab5),label=above:$A_{2}$,text width=1.45cm] {};
	\node [fit=(ca1) (ca3),label=above:$B_{2,1}$,text width=0.5cm,pattern = north east lines] {};
	\node [fit=(ac1) (ac5),label=above:$A_{3}$,text width=1.45cm] {};
	\node [fit=(da1) (da3),label=above:$B_{3,1}$,text width=0.5cm,pattern = north east lines] {};			

	\foreach \i in {1,2,3}
	\foreach \j in {1,2,3}
	{\draw (ba\i) to[bend right=50] (bb\j);}

	\foreach \i in {1,2,...,5}
\foreach \j in {1,2,...,3}
{\draw (aa\i) to[bend right=50] (ba\j);
\draw (ab\i) to[bend right=50] (ca\j);
\draw (ac\i) to[bend right=50] (da\j);}

	\foreach \i in {1,2,...,3}
\foreach \j in {1,2,...,5}
{\draw (bb\i) to[bend right=50] (ab\j);
	\draw (ca\i) to[bend right=50] (ac\j);
}
\end{tikzpicture}
\captionsetup{justification=centering}
\caption{An example of an $(s,t)$-chain. Here the chain has length $k=3$, the striped regions $B_{i,j}$ form cliques of size $t=3$, and the sets $A_i$ have size at least $\left(\binom{s+t}{s}\right)^{10}$.}    
\label{figure_chain}
\end{figure}
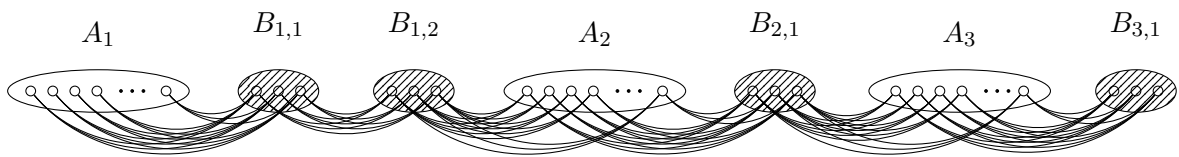

\begin{lemma} \label{lem:connectedobject}
    Let $s,t\geq 2$. Let $N\geq \left(\binom{s+t}{s}\right)^{300}$ and let $G$ be a vertex-ordered graph on $N$ vertices without an independent set of size $s$. Then there exists an $(s,t)$-chain of length $s$ in $G$.
\end{lemma}

Lemma~\ref{lem:connectedobject} is proved in Sections~\ref{section_slarge} and~\ref{section_ssmall}, dealing with the cases $s>t^2$ and $s\leq t^2$ respectively. 
We will also need the following lemma, which is similar to Lemma 4.3 in \cite{gishboliner2023ramsey}.
\begin{lemma}\label{lemma_sequence}
    Let $s$, $n$ and $N$ be positive integers with $N\geq sn$, and let $f_1,\ldots ,f_{s-1}$ be functions from $[N]$ to $[n]$. 
    Then there exist $x_1<x_2<\ldots< x_s$ in $[N]$ such that $f_i(x_i)\geq f_i(x_{i+1})$ for all $1\leq i\leq s-1$. 
\end{lemma}
\begin{proof}
    The proof is by induction on $s$. 
    For $s=1$ there is nothing to show. Now, let $y_1=1$, and for $i\geq 2$ let $y_i$ be the smallest element in $[y_{i-1}+1,N]$ such that $f_1(y_i)>f_1(y_{i-1})$ (if such a $y_i$ exists). This gives a sequence of order at most $n$, say $y_1<\dots <y_k$ (with $k\leq n$). Applying induction for the functions $f_2,\ldots ,f_{s-1}$ on the set $[N]\setminus \{y_1,\dots,y_k\}$, we find $x_2<x_3<\ldots <x_s$ such that $x_i\in [N]\setminus\{y_1,\dots,y_k\}$ for all $i\in [2,s]$ and $f_i(x_i)\geq f_i(x_{i+1})$ for all $i\in [2,s-1]$.
    
    Let $j\in[k]$ be maximal with $y_j<x_2$, and set $x_1=y_j$. We claim that we have $f_1(x_1)\geq f_1(x_2)$. Indeed, if we had $f_1(x_1)< f_1(x_2)$, then (by the definition of the sequence $y_i$) we would have ($j<k$ and) $y_{j+1}\leq x_2$ and hence $y_{j+1}< x_2$, contradicting the choice of $j$. Thus, we do indeed have $f_1(x_1)\geq f_1(x_2)$, and hence the sequence $x_1,\dots,x_s$ satisfies the required conditions.
\end{proof}

We are now ready to prove Theorem~\ref{thm:pathvscliquegeneral} (subject to the proofs of Lemma~\ref{lemma_findKt+tfroms} and Lemma~\ref{lem:connectedobject}).

\begin{proof}[Proof of Theorem~\ref{thm:pathvscliquegeneral}]
    Let $N=\left(\binom{s+t}{s}\right)^{300}sn$. (A straightforward but tedious computation shows that then $N<R(s,t)^{C}n$ for some absolute constant $C$.) We will show that if $G$ is a graph on vertex set $[N]$ without an independent set of size $s$, then $G$ contains a $P_n^t$. For each $a\in[sn]$, write $X_a=\left[(a-1)\left(\binom{s+t}{s}\right)^{300}+1,a\left(\binom{s+t}{s}\right)^{300}\right]$, so that $|X_a|=\left(\binom{s+t}{s}\right)^{300}$ and $X_1<\dots<X_{sn}$. By Lemma~\ref{lem:connectedobject}, for each $i\in[sn]$, we can find an $(s,t)$-chain of length $s$ in $G[X_a]$, say $A_1^{(a)}<B_{1,1}^{(a)}<\dots <B_{1,p_1^{(a)}}^{(a)}<A_2^{(a)}<B_{2,1}^{(a)}<\dots<B_{2,p_2^{(a)}}^{(a)}<\dots<A_s^{(a)}<B_{s,1}^{(a)}<\dots<B_{s,p_s^{(a)}}^{(a)}$.

    For each $i\in[s-1]$ and $a\in[sn]$, let $f_i(a)$ be the length of the longest (monotone) $t$-th power of a path whose last $t$ vertices belong to $A_{s-i}^{(a)}$. (Note that $|A^{(a)}_{s-i}|>R(t,s)$, so $A^{(a)}_{s-i}$ contains a $t$-clique and thus $f_i(a)$ is well-defined.) If $f_i(a)\geq n$ for some $i\in[s-1]$ and $a\in[sn]$ then we are done, so we may assume that $f_1,\dots,f_{s-1}$ are functions from $[sn]$ to $[n]$. Then, by Lemma~\ref{lemma_sequence}, there exists an increasing sequence $a_1<\dots<a_s$ in $[sn]$ such that $f_i(a_i)\geq f_i(a_{i+1})$ for all $i\in[s-1]$.

    For all $a\in[sn]$ and $i\in [s]$, let us write $\bar{B}_i^{(a)}=B_{i,1}^{(a)}\cup B_{i,2}^{(a)}\cup\dots\cup B_{i,p_i^{(a)}}^{(a)}$. Note that, for all $a\in [sn]$ and $i\in[s-1]$, if $C$ is a clique of size $t$ in $A_{s-i+1}^{(a)}$ and a set $S\subseteq V(G)$ gives a (monotone) $t$-th power of a path ending in $t$ vertices of $A_{s-i}^{(a)}$, then $S\cup \bar{B}_{s-i}^{(a)}\cup C$ also gives a monotone $t$-th power of a path, ending in $t$ vertices of $A_{s-i+1}^{(a)}$. In particular, since $|A_{s-i+1}^{(a)}|>R(t,s)$, we have $f_{i-1}(a)>f_{i}(a)$ for all $a\in[sn]$ and $i\in[2,s-1]$, and hence $f_1(a_1)>f_2(a_2)>\dots>f_{s-1}(a_{s-1})$.

    For each $i\in[s]$, let $V_i=A_{s-i+1}^{(a_i)}$. By Lemma~\ref{lemma_findKt+tfroms}, there exist some $1\leq i<j\leq s$ and subsets $T_i\subseteq V_i$ and $T_j\subseteq V_j$ of size $t$ such that $T_i\cup T_j$ is a clique in $G$. By the definition of $f_i$, there is a $P_{f_i(a_i)}^t$ on some vertex set $S$ such that the last $t$ vertices of $S$ belong to $A_{s-i}^{(a_i)}$. But then $S\cup \bar{B}_{s-i}^{(a_i)}\cup T_i\cup T_j$ gives a $t$-th power of a path of length more than $f_i(a_i)$ and ending in $t$ vertices in $A_{s-j+1}^{(a_j)}$, whence $f_{j-1}(a_j)>f_i(a_i)$. However, we know that $f_i(a_i)>f_{i+1}(a_{i+1})>\dots>f_{j-1}(a_{j-1})\geq f_{j-1}(a_j)$, giving a contradiction.
\end{proof}

\begin{remark}
    A variant of the argument presented above can also be used to give a very short proof of the bound $R_<(P_n^t,K_s)\leq s^{O(st)}n$ that was obtained in \cite{gishboliner2023ramsey}. Indeed, let us sketch a proof of the bound $R_<(P_n^t,K_s)\leq R(sr,s)sn$, where $r=(\binom{s+t}{s})^{10}$. (It is easy to see that this indeed implies $R_<(P_n^t,K_s)\leq s^{O(st)}n$.) Let $N=R(sr,s)sn$ and let $G$ be a graph on vertex set $[N]$ without an independent set of size $s$. Clearly, each interval of length $R(sr,s)$ contains a clique of size $sr$, so we may find sets $X_1<X_2<\dots<X_{sn}$ of size $sr$, each forming a clique. Partition each $X_a$ into $s$ sets $A_1^{(a)}<A_2^{(a)}<\dots<A_s^{(a)}$ of size $r$. Then we may use the same argument as in the proof of Theorem \ref{thm:pathvscliquegeneral}, with the only difference that the sets $B_{i,j}$ are not needed as $A_1^{(a)}\cup A_2^{(a)}\cup \dots \cup A_s^{(a)}$ is a clique.
\end{remark}

\subsection{The proof of Lemma \ref{lemma_findKt+tfroms}} \label{sec:auxlemma}

In this section we prove Lemma \ref{lemma_findKt+tfroms}. We will need the following result here as well as in Section~\ref{section_ssmall}.

\begin{lemma} \label{lem:findKt+t}
    Let $s,t\geq 2$. Let $V_1$ and $V_2$ be disjoint sets of size at least $(\binom{s+t}{s})^{8}$. Let $G$ be a graph with vertex set $V_1\cup V_2$ such that the density of edges between $V_1$ and $V_2$ is at least $(\frac{t}{s+t})^2$ and $G$ has no independent set of size $s$. Then there exist some subsets $T_1\subset V_1$ and $T_2\subset V_2$ of size $t$ such that $T_1\cup T_2$ is a clique in $G$.
\end{lemma}

\begin{proof}
    Without loss of generality, assume that $|V_2|\geq |V_1|$. Let $N=|V_1|$. Let us choose $2t$ vertices in $V_2$ at random with replacement, and write $T$ for the set of these vertices. Let $R$ be the common neighbourhood of $T$ inside $V_1$. Call a subset of $R$ of size $t$ \emph{bad} if its common neighbourhood inside $V_2$ has size at most $N^{1/2}$. Then the expected number of bad sets of size $t$ in $R$ is at most $N^t\cdot (\frac{N^{1/2}}{|V_2|})^{2t}\leq N^t\cdot (N^{-1/2})^{2t}=1$. Delete one vertex from each bad $t$-set in $R$ to obtain a subset $R'$. Note that $\mathbb{E}[|R'|]\geq \mathbb{E}[|R|]-1$. On the other hand, since the density of edges between $V_1$ and $V_2$ is at least $(\frac{t}{s+t})^2$, by Jensen's inequality, $\mathbb{E}[|R|]\geq ((\frac{t}{s+t})^{2})^{2t} N=(\frac{t}{s+t})^{4t} N$. Moreover, $\binom{s+t}{t}\geq (\frac{s+t}{t})^{t}$, so $N\geq (\frac{s+t}{t})^{8t}$ and therefore $\mathbb{E}[|R|]\geq N^{1/2}$. It follows that there exists an outcome in which $|R'|\geq N^{1/4}$. Since $N^{1/4}\geq R(s,t)$ and $G$ has no independent set of size $s$, $R'$ contains a clique $T_1$ of size $t$. By the construction of $R'$, $T_1$ is not a bad set, so its common neighbourhood inside $V_2$ has size at least $N^{1/2}$. Writing $S$ for this common neighbourhood, we can find a clique $T_2$ of size $t$ in $S$ (since $N^{1/2}\geq R(s,t)$ and $G$ has no independent set of size $s$). Then $T_1\cup T_2$ is a clique, completing the proof.
\end{proof}

\begin{proof}[Proof of Lemma \ref{lemma_findKt+tfroms}]
    We use induction on $s$. The statement is trivial for $s=2$. Now assume that there exists some $v\in V_s$ such that $v$ sends at most $\frac{t}{s+t} N$ edges to each $V_i$, $1\leq i\leq s-1$. For each $1\leq i\leq s-1$, let $V_i'$ be the set of non-neighbours of $v$ inside $V_i$. Note that $|V_i'|\geq \frac{s}{s+t}\cdot (\binom{s+t}{s})^{10}\geq (\binom{s+t-1}{s-1})^{10}$. Observe that $G[V_1'\cup \dots \cup V_{s-1}']$ has no independent set of size $s-1$. Hence, by the induction hypothesis, there exist some $1\leq i< j\leq s-1$ and subsets $T_i\subset V_i'$ and $T_j\subset V_j'$ of size $t$ such that $T_i\cup T_j$ is a clique in $G$.

    So assume that for every $v\in V_s$, there exists some $1\leq i\leq s-1$ such that $v$ sends at least $\frac{t}{s+t} N$ edges to $V_i$. Without loss of generality, we may assume that there exists $V_{s}'\subset V_s$ of size at least $N/s\geq (\binom{s+t}{s})^{8}$ such that each $v\in V_s'$ sends at least $\frac{t}{s+t}N$ edges to $V_1$.
    Then, by Lemma~\ref{lem:findKt+t} applied to the sets $V_1$ and $V_s'$, there exist subsets $T_1\subset V_1$ and $T_s\subset V_s'$ of size $t$ such that $T_1\cup T_s$ is a clique in $G$.
\end{proof}

\subsection{When $s$ is large compared to $t$}\label{section_slarge}
In this section we prove Lemma~\ref{lem:connectedobject} in the case $s>t^2$. Note that, since $\binom{s+t}{s}\geq \left(\frac{s+t}{t}\right)^t\geq \left(s^{1/2}\right)^t$, it is enough to prove the following lemma.
\begin{lemma} \label{lem:slarge}
    \sloppy Let $t\geq 2$ and $s>t^2$ be positive integers. Let $N\geq s^{150t}$ and let $G$ be a vertex-ordered graph on $N$ vertices without an independent set of size $s$. Then there exists an $(s,t)$-chain of length $s$ in $G$.
\end{lemma}

\begin{proof}
Given a positive integer $\ell$ and a subgraph $H$ of $G$, let  $m_{\ell}(H)$ denote the number of cliques of size $\ell$ in $H$.
	\begin{claim}\label{claim_findingell}
		There exist some $\ell\in[t+2,3t+1]$ and an induced subgraph $G'$ of $G$ with $|G'|\geq N/s^{6t}$ such that if $H$ is an induced subgraph of $G'$ with $|H|\geq |G'|/s^{3}$, then 
		$m_\ell(H)\geq m_{\ell-1}(G')\cdot N/s^{25t}$.
	\end{claim}
	
	\begin{proof}
		We define a sequence of induced subgraphs $G_i$ of $G$ with $|G_i|\geq N/s^{3i}$. We start with $G_0=G$. Having constructed $G_i$, if $G_i$ has an induced subgraph $H$ with $|H|\geq |G_i|/s^{3}$ and $m_{t+i+2}(H)< m_{t+i+1}(G_i)\cdot N/s^{25t}$, then put $G_{i+1}=H$.  
		Otherwise terminate the procedure. Clearly, the claim follows if the process terminates after less than $2t$ steps, as the final $G_i$ will satisfy the conditions with $\ell=t+i+2$. On the other hand, if the process runs for at least $2t$ steps, then $m_{3t+1}(G_{2t})\leq m_{t+1}(G) N^{2t}/s^{50t^2}\leq N^{3t+1}/s^{50t^2}$. Since $|G_{2t}|\geq N/s^{6t}>2R(3t+1,s)$, and all induced subgraphs of $G_{2t}$ of order $R(3t+1,s)$ contain a $K_{3t+1}$, we have, by double counting, 
		\begin{align*}
			m_{3t+1}(G_{2t})\geq \frac{\binom{|G_{2t}|}{R(3t+1,s)}}{\binom{|G_{2t}|}{R(3t+1,s)-3t-1}}\geq \left(\frac{|G_{2t}|/2}{R(3t+1,s)}\right)^{3t+1}\geq\left(\frac{N}{(3s)^{3t+1} s^{6t}}\right)^{3t+1} >\frac{N^{3t+1}}{s^{50t^2}},
		\end{align*}
		giving a contradiction.
	\end{proof}

 Let $G'$ and $\ell$ be as in the claim above, and let $N'=|G'|\geq N/s^{6t}$. Without loss of generality, $G'$ has vertex set $[N']$ (with the usual ordering). Given a subset $S\subseteq [N']$ of size $\ell$ which induces a clique, let  $\chi(S)$ be the largest $k$ such that there is an $(s,\ell-2)$-chain of length $k$ in $G'$ ending in the middle $\ell-2$ vertices of $S$ as the final $(\ell-2)$-clique $B_{k,p_k}$ of the chain. (If no such $(s,\ell-2)$-chain ends in those $\ell-2$ vertices, we put $\chi(S)=0$.) If there is some $S$ with $\chi(S)\geq s$, we are done, now assume that no such $S$ exists. We will show that we can find many $S,S'$ (both inducing a clique of order $\ell$) such that $\max(S)=\min(S')$ and $\chi(S)\geq\chi(S')$ -- we will call such a pair of cliques $(S,S')$ a \emph{good pair} -- and use this to obtain a contradiction.

Given an induced subgraph $H$ of $G'$, let us write $g(H)$ for the number of good pairs in $G'$ which are contained in $H$. Note that if $S$ induces a $K_\ell$ in $H$ with $\chi(S)=i$, the corresponding $(s,\ell-2)$-chain of length $i$ (ending in the middle $\ell-2$ vertices of $S$) may have vertices outside $H$. The following claim shows that we cannot have too many good pairs contained in $H$.
	
	\begin{claim}\label{claim_numberofgoodpairs}
		If $H$ is an induced subgraph of $G'$, then $g(H)< 
		m_{\ell-1}(H)^2\cdot (4s)^{10\ell}$.
	\end{claim}
	\begin{proof}
		Assume that $g(H)\geq m_{\ell-1}(H)^2\cdot (4s)^{10\ell}$. Then there exist two disjoint sets $S,S'$ in $V(H)$, both of size $\ell-1$ and inducing a clique, such that $\max(S)<\min(S')$ and the set
		\begin{align*}
		    Z=\{z\in V(H): &\max(S)<z<\min(S'), \\&S\cup\{z\} \textnormal{ and } \{z\}\cup S'\textnormal{ are cliques with } \chi(S\cup\{z\})\geq\chi(\{z\}\cup S')\}
		\end{align*}
		has size at least $(4s)^{10\ell}\geq \left(\binom{s+\ell-2}{s}\right)^{10}$. It follows that any $(s,\ell-2)$-chain ending at the last $\ell-2$ vertices of $S$ can be extended to a longer $(s,\ell-2)$-chain ending at the first $\ell-2$ vertices of $S'$. But then, taking an arbitrary $z\in Z$, we have  $\chi(\{z\}\cup S')>\chi(S\cup\{z\})$, which is a contradiction.
	\end{proof}
	
	We now use Claim~\ref{claim_numberofgoodpairs} to prove the following claim, which will then be used iteratively (about $\log_2 s$ times) to pass to large induced subgraphs of $G'$ where most of the $\ell$-cliques $S$ have $\chi(S)$ coming from an increasingly small set of values in $[0,s-1]$.
	\begin{claim}\label{claim_2/5}
		Let $H$ be an induced subgraph of $G'$ and let $I\subseteq [0,s-1]$. Then there is an induced subgraph $H'$ of $H$ and some $J\subseteq I$ such that $|H'|\geq \frac{2}{5}|H|$, $|J|\geq \lceil |I|/2\rceil$, and $H'$ contains at most $3 (4s)^{5\ell}N^{1/2}m_{\ell-1}(H)$ subsets $S$ of size $\ell$ inducing a clique and having $\chi(S)\in J$.		
	\end{claim}
	\begin{proof}
		Let $J_1$ be the set of $\lceil |I|/2\rceil$ smallest elements of $I$, and let $J_2$ be the largest $\lceil |I|/2\rceil$ elements of $I$. Note that $\max(J_1)\leq \min(J_2)$. We may assume that $|H|\not =0$. Let $q=\frac{3 (4s)^{5\ell}}{|H|^{1/2}}m_{\ell-1}(H)$, let $$X_1=\{x\in V(H): x\textnormal{ is the first vertex of at least } q \textnormal{ cliques } S \textnormal{ in } H \textnormal{ of order } \ell \textnormal{ with } \chi(S)\in J_1\}$$
		and
		$$X_2=\{x\in V(H): x\textnormal{ is the last vertex of at least } q \textnormal{ cliques } S \textnormal{ in } H \textnormal{ of order } \ell \textnormal{ with } \chi(S)\in J_2\}.$$
		
		By taking $H'=H[V(H)\setminus X_i]$ and $J=J_i$, we are done if $|X_i|\leq \frac{3}{5}|H|$ for some $i\in \{1,2\}$, since then $H'$ contains at most $|H'|q\leq 3 (4s)^{5\ell}N^{1/2}m_{\ell-1}(H)$ cliques $S$ of order $\ell$ with $\chi(S)\in J$. On the other hand, if $|X_1|\geq\frac{3}{5}|H|$ and $|X_2|\geq \frac{3}{5}|H|$, then $|X_1\cap X_2|\geq |H|/5$, which immediately gives
		$$g(H)\geq \frac{|H|}{5}q^2=\frac{|H|}{5}\cdot \left(\frac{3 (4s)^{5\ell}}{|H|^{1/2}}m_{\ell-1}(H)\right)^2\geq m_{\ell-1}(H)^2\cdot (4s)^{10\ell},$$
		contradicting Claim~\ref{claim_numberofgoodpairs}.
	\end{proof}
		We are now ready to finish the proof of our lemma. We define a sequence of graphs $H_i$ and sets $I_i\subseteq [0,s-1]$ as follows. Set $H_0=G'$, and let $I_0=[0,s-1]$. Having defined $H_i$ and $I_i$, we apply Claim~\ref{claim_2/5} to $(H_i,I_i)$ to get some $H'$ and $J$, and we set $H_{i+1}=H'$, $I_{i+1}=I_i\setminus J$. We terminate this procedure when we get $I_i=\emptyset$ for some $i$. Let $j$ be the final value of $i$, so that $I_j=\emptyset$.
	
	Note that for all $i<j$ we have $|I_{i+1}|\leq |I_i|/2$, so the process terminates in $j\leq \log_2(s)+1$ steps. As $|H_{i+1}|\geq \frac{2}{5}|H_i|$ for all $i<j$, we get
	$$|H_j|\geq \left(\frac{2}{5}\right)^{j}|H_0|> \left(\frac{1}{4}\right)^{\log_2 s+1}|H_0|=\frac{1}{4s^2}|H_0|>\frac{1}{s^3}|H_0|.$$
	
	Hence, by the definition of $H_0$, we have
    \begin{equation}
        m_{\ell}(H_j)\geq  m_{\ell-1}(H_0)\cdot N/s^{25t}. \label{eqn:lower bound for m}
    \end{equation}
	On the other hand, by the definition of the sets $I_i$, for each $i<j$ we have that $H_j$ contains at most $3(4s)^{5\ell}N^{1/2}m_{\ell-1}(H_0)$ cliques $S$ of size $\ell$ with $\chi(S)\in I_i\setminus I_{i+1}$. Since $\bigcup_{i=0}^{j-1}(I_{i}\setminus I_{i+1})=[0,s-1]$ and $j\leq s$, we get that
	\begin{equation}
		m_{\ell}(H_j)\leq s\cdot 3(4s)^{5\ell}N^{1/2}m_{\ell-1}(H_0). \label{eqn:upper bound for m}
	\end{equation}
	Comparing equations (\ref{eqn:lower bound for m}) and (\ref{eqn:upper bound for m}) (and noting that $m_{\ell-1}(H_0)\not =0$ as $|H_0|>R(\ell-1,s)$), we get $N/s^{25t}\leq (4s)^{5\ell+1}N^{1/2}<s^{10\ell+2}N^{1/2}$, whence
	$$N\leq s^{50t+20\ell+4}<s^{150t},$$
	giving a contradiction.
\end{proof}

\subsection{When $s$ is not too large compared to $t$}\label{section_ssmall}
To complete the proof of Lemma~\ref{lem:connectedobject}, it is left to deal with the case $s\leq t^2$, which follows from the following lemma by taking $\ell=r=s$.

\begin{lemma}
Let $s$ and $t$ be positive integers with $s\leq t^2$, and let $\ell,r$ be positive integers with $\ell\leq s$ and $2\leq r\leq s$. Let $N= \left(\binom{t+s}{t}\binom{t+\ell}{t}\binom{t+r}{t}\right)^{50}$. Then whenever $G$ is a graph on vertex set $[N]$ without an independent set of size $r$, then there exists an $(s,t)$-chain of length $\ell$ in $G$.
\end{lemma}
\begin{proof}
    We prove the statement by induction on $\ell+r$. For the case $\ell\leq 3$, we will need the following claim. As before, given a positive integer $k$, $m_k(G)$ denotes the number of cliques of size $k$ in $G$.

	\begin{claim}\label{claim_findingell2}
		There exists some $k\in[2t,5t-3]$ such that $m_{k+3}(G)> m_{k}(G)\cdot N^2(\binom{s+t}{s})^{10}$.
	\end{claim}
 \begin{proof}
If the claim fails, then we have
\begin{align*}
    m_{5t}(G)&\leq m_{2t}(G)\left(N^2\left(\binom{s+t}{s}\right)^{10}\right)^{t}\leq N^{4t}\left(\binom{s+t}{s}\right)^{10t}.
\end{align*}
Note that
$R(5t,s)\leq \binom{5t+s}{5t}\leq \left(\frac{t+s}{t}\right)^{4t}\binom{t+s}{t}\leq \left(\binom{t+s}{t}\right)^5<N/2$. Thus, by double counting,
\begin{align*}m_{5t}(G)\geq \frac{\binom{N}{R(5t,s)}}{\binom{N}{R(5t,s)-5t}}\geq \left(\frac{N/2}{R(5t,s)}\right)^{5t}\geq \left(\frac{N}{\left(\binom{t+s}{t}\right)^6}\right)^{5t}\geq N^{4t}\left(\binom{t+s}{t}\right)^{50t-30t},
		\end{align*}
		giving a contradiction.
	\end{proof}

    We now show that the statement of the lemma holds whenever $\ell\leq 3$. Pick $k$ as in Claim~\ref{claim_findingell2}. Since  $m_{k+3}(G)> m_{k}(G)\cdot N^2(\binom{s+t}{s})^{10}$, there exists a set $S$ of size $k$, inducing a clique, such that, writing $S_1$ for the set of $t$ smallest elements of $S$, there are at least $N^2(\binom{s+t}{s})^{10}$ cliques of size $k+3$ of the form $\{x,y,z\}\cup S$ with $\{x\}<S_1<\{y\}<(S\setminus S_1)<\{z\}$. Hence, there exist sets $X,Y,Z$, each of size at least $(\binom{s+t}{s})^{10}$, such that $X<S_1<Y<(S\setminus S_1)<Z$, and $X\cup Y\cup Z$ is complete to $S$. Taking an arbitrary subset $S_2\subseteq (S\setminus S_1)$ of size $t$, $X<S_1<Y<S_2<Z$ gives us an $(s,t)$-chain of length $3$, as claimed.

    Now assume that $\ell>3$.
    If $r=2$, then the statement holds trivially, so we may assume $r\geq 3$. Let $A=[\frac{t}{s+t}N]$ and $B=[N]\setminus A$. First consider the case when there is some vertex $v\in A$ such that there is a set $X\subseteq B$ with $|X|\geq \left(\frac{s}{s+t}\right)^2 N$ and $N(v)\cap X=\emptyset$. Then \begin{align*}
        |X|&\geq \left(\binom{t+s}{t}\binom{t+\ell}{t}\binom{t+r}{t}\right)^{50}\left(\frac{s}{s+t}\right)^2\\
        &\geq \left(\binom{t+s}{t}\binom{t+\ell}{t}\binom{t+r}{t}\right)^{50}\left(\frac{r}{r+t}\right)^2\\
        &\geq \left(\binom{t+s}{t}\binom{t+\ell}{t}\binom{t+r-1}{t}\right)^{50}.
    \end{align*}
    Moreover, $G[X]$ has no independent set of size $r-1$. Hence, by induction, $G[X]$ contains an $(s,t)$-chain of length $\ell$, as required.

    Thus, from now on, we may assume that for all $v\in A$, $|N(v)\cap B|\geq \frac{st}{(s+t)^2}N$.
    We will use the following inequality twice.

    \begin{claim} \label{claim:inequality}
        We have $$\frac{s^2t}{(s+t)^3}\left(\binom{t+\ell}{t}\right)^{50}\geq \left(\binom{t+\lceil\frac{2}{3}\ell\rceil}{t}\right)^{50}.$$
    \end{claim}

    \begin{proof}
    Using $\ell\leq s$ and $s\leq t^2$, we have
        \begin{align*}
    &\frac{s^2t}{(s+t)^3} \left(\binom{t+\ell}{t}\right)^{50} \geq \left(\frac{\ell}{\ell+t}\right)^2\frac{t}{t^2+t} \left(\left(\frac{t+\ell}{\ell}\right)^{\lfloor \ell/3\rfloor}\binom{t+\lceil \frac{2}{3}\ell\rceil}{t}\right)^{50}\\
    &\geq\frac{1}{t+1}\left(1+\frac{t}{\ell}\right)^{\ell} \left(\binom{t+\lceil \frac{2}{3}\ell\rceil}{t}\right)^{50}\geq \left(\binom{t+\lceil \frac{2}{3}\ell\rceil}{t}\right)^{50},
    \end{align*}
    which proves the claim.
    \end{proof}
    
    Using Claim \ref{claim:inequality} (and since $\frac{t}{s+t}\geq \frac{s^2t}{(s+t)^3}$), we have 
    \begin{align*}
    |A|&=\frac{t}{s+t} \left(\binom{t+s}{t}\binom{t+\ell}{t}\binom{t+r}{t}\right)^{50}\geq \left(\binom{t+s}{t}\binom{t+\lceil\frac{2}{3}\ell\rceil}{t}\binom{t+r}{t}\right)^{50}.
    \end{align*}
    Hence, by induction, $G[A]$ contains an $(s,t)$-chain of length $\lceil \frac{2}{3}\ell\rceil$, say $X_1<Y_{1,1}<\dots <Y_{1,p_1}<X_2<Y_{2,1}<\dots<Y_{2,p_2}<\dots<X_{\lceil \frac{2}{3}\ell\rceil}<Y_{\lceil \frac{2}{3}\ell\rceil,1}<\dots<Y_{\lceil \frac{2}{3}\ell\rceil,p_{\lceil \frac{2}{3}\ell\rceil}}$. Let $Z$ denote the set of vertices in $B$ which have at least $\left(\frac{t}{s+t}\right)^2|X_{\lceil \frac{2}{3}\ell\rceil}|$ neighbours in $X_{\lceil \frac{2}{3}\ell\rceil}$. Note that
    $$|X_{\lceil \frac{2}{3}\ell\rceil}|\cdot\frac{st}{(s+t)^2}N\leq e(X_{\lceil \frac{2}{3}\ell\rceil},B)\leq |X_{\lceil\frac{2}{3}\ell\rceil}||Z|+\left(\frac{t}{s+t}\right)^2|X_{\lceil \frac{2}{3}\ell\rceil}|\left(\frac{s}{s+t}N-|Z|\right),$$
    giving
    $$|Z|\geq \frac{\frac{s^2t}{(s+t)^3}N}{1-\frac{t^2}{(s+t)^2}}\geq \frac{s^2t}{(s+t)^3}N.$$

    Thus, by Claim \ref{claim:inequality},
    \begin{align*}
    |Z|&\geq\frac{s^2t}{(s+t)^3} \left(\binom{t+s}{t}\binom{t+\ell}{t}\binom{t+r}{t}\right)^{50}\geq \left(\binom{t+s}{t}\binom{t+\lceil \frac{2}{3}\ell\rceil}{t}\binom{t+r}{t}\right)^{50}.
    \end{align*}
    By induction, $G[Z]$ contains an $(s,t)$-chain of length $\lceil \frac{2}{3}\ell\rceil$, say $X_1'<Y_{1,1}'<\dots <Y_{1,q_1}'<X_2'<Y_{2,1}'<\dots<Y_{2,q_2}'<\dots<X_{\lceil \frac{2}{3}\ell\rceil}'<Y_{\lceil \frac{2}{3}\ell\rceil,1}'<\dots<Y_{\lceil \frac{2}{3}\ell\rceil,q_{\lceil \frac{2}{3}\ell\rceil}}'$.

    By the definition of $(s,t)$-chains, we know that $|X_{\lceil \frac{2}{3}\ell\rceil}|,|X_1'|\geq \left(\binom{s+t}{s}\right)^{10}$. Moreover, by the definition of $Z$, the edge density between $X_{\lceil \frac{2}{3}\ell\rceil}$ and $X_1'$ is at least $\left(\frac{t}{s+t}\right)^2$. Hence, by Lemma~\ref{lem:findKt+t}, there are subsets $Y\subseteq X_{\lceil \frac{2}{3}\ell\rceil}$ and $Y'\subseteq X_1'$, each of size $t$, such that $G[Y\cup Y']$ is complete. But then $X_1<Y_{1,1}<\dots <Y_{1,p_1}<X_2<Y_{2,1}<\dots<Y_{2,p_2}<X_3<\dots<Y_{\lceil \frac{2}{3}\ell\rceil-1,p_{\lceil \frac{2}{3}\ell\rceil-1}}<Y<Y'<Y_{1,1}'<\dots <Y_{1,q_1}'<X_2'<Y_{2,1}'<\dots<Y_{2,q_2}'<\dots<X_{\lceil \frac{2}{3}\ell\rceil}'<Y_{\lceil \frac{2}{3}\ell\rceil,1}'<\dots<Y_{\lceil \frac{2}{3}\ell\rceil,q_{\lceil \frac{2}{3}\ell\rceil}}'$ gives an $(s,t)$-chain of length $2\left(\lceil\frac{2}{3}\ell\rceil-1\right)\geq \ell$, finishing the proof.
\end{proof}

\section{Concluding remarks and open problems} \label{sec:remarks}

Answering a question of Gishboliner, Jin and Sudakov \cite{gishboliner2023ramsey}, we showed that the ordered Ramsey number of the $t$-th power of a path on $n$ vertices is at most $n^{4+o(1)}$ for any fixed $t$. 
The trivial lower bound gives $R_<(P_n^t,P_n^t)=\Omega(n^2)$. It would be interesting to find the correct value of the exponent. We believe that the truth is closer to the lower bound. 
\begin{problem}
    Is $R_<(P^t_n,P_n^t)=n^{2+o(1)}$ as $n\rightarrow \infty$ ?  
\end{problem}

Let us write $R_<^r(P_n^t)$ for the $r$-colour (ordered) Ramsey number of $P_n^t$. The proof of Theorem~\ref{theorem_pathvspath} can be straightforwardly generalised to show the following. 

\begin{theorem} \label{thm:multicolour}
    Let $r\geq 2$ and $t\geq 1$ be integers. Then there exist an absolute constant $C$ and a constant $D=D(r,t)$ such that $R^r_<(P_n^t)\leq D n^{C r\log r}$. 
\end{theorem}

This result together with an observation in \cite{Ramseysparsedigraphs} gives a tournament version of Theorem~\ref{theorem_pathvspath}.
Let $\overset{\rightarrow}{P}_n^t$ be the directed graph which is the $t$-th power of the directed path on $n$ vertices. 

\begin{theorem} \label{thm:tournament}
   Let $r\geq 2$ and $t\geq 1$ be integers. Then there exist an absolute constant $C'$ and a constant $D'=D'(r,t)$ such that every $r$-edge coloured tournament on $N$ vertices contains a monochromatic $\overset{\rightarrow}{P}_n^t$, provided that $N\geq D' n^{C'r\log r}$. 
\end{theorem}

Theorem \ref{thm:tournament} is already interesting for $r=2$. We remark that the problems of finding monochromatic directed paths in multicoloured tournaments and finding powers of paths in (uncoloured) tournaments have both been studied (see, e.g., \cite{chvatal1972monochromatic,gyarfas1973ramsey,draganic2021powers}).

To see how Theorem \ref{thm:tournament} follows from Theorem \ref{thm:multicolour}, let $C$ and $D'=D(2r,t)$ be constants chosen according to Theorem \ref{thm:multicolour}, and let $T$ be a tournament on $N\geq D' n^{C(2r)\log (2r)}$ vertices with an $r$-edge colouring $\chi$. Fix an arbitrary ordering of $V(T)$, say $[N]$. 
Now, given $x<y$ colour $(x,y)$ by $(\chi(xy),1)$ if $x\rightarrow y$ and by $(\chi(xy),2)$ otherwise. 
This is a $2r$-colouring of $[N]^{(2)}$. We may now invoke Theorem~\ref{thm:multicolour} to find a monochromatic ordered $P_n^t$ which corresponds to a monochromatic $\overset{\rightarrow}{P}_n^{t}$. So we may take $C'=4C$.

Another interesting problem is to better estimate the $r$-colour ordered Ramsey number of $P_n^t$. 
We conjecture the following. 
\begin{conjecture}
    For all $r,t,n$, $R_<^r(P^t_n)=O_{r,t}\left (n^{O(r)}\right)$.
\end{conjecture}

\paragraph{Acknowledgements.} We are grateful to Alex Scott for bringing the problems studied in this paper to our attention, and for inviting the third author to Oxford.

\bibliographystyle{abbrv}
\bibliography{Bibliography}

\begin{thebibliography}{10}

\bibitem{balko2015ramsey}
M.~Balko, J.~Cibulka, K.~Kr{\'a}l, and J.~Kyn{\v{c}}l.
\newblock Ramsey numbers of ordered graphs.
\newblock {\em Electronic Journal of Combinatorics}, 27(1):P1.16, 2020.

\bibitem{chvatal1972monochromatic}
V.~Chv{\'a}tal.
\newblock Monochromatic paths in edge-colored graphs.
\newblock {\em Journal of Combinatorial Theory, Series B}, 13(1):69--70, 1972.

\bibitem{conlon2017ordered}
D.~Conlon, J.~Fox, C.~Lee, and B.~Sudakov.
\newblock Ordered {R}amsey numbers.
\newblock {\em Journal of Combinatorial Theory, Series B}, 122:353--383, 2017.

\bibitem{draganic2021powers}
N.~Dragani{\'c}, F.~Dross, J.~Fox, A.~Gir{\~a}o, F.~Havet, D.~Kor{\'a}ndi,
  W.~Lochet, D.~M. Correia, A.~Scott, and B.~Sudakov.
\newblock Powers of paths in tournaments.
\newblock {\em Combinatorics, Probability and Computing}, 30(6):894--898, 2021.

\bibitem{erdos1935combinatorial}
P.~Erd{\H{o}}s and G.~Szekeres.
\newblock A combinatorial problem in geometry.
\newblock {\em Compositio mathematica}, 2:463--470, 1935.

\bibitem{Ramseysparsedigraphs}
J.~Fox, X.~He, and Y.~Wigderson.
\newblock Ramsey numbers of sparse digraphs.
\newblock {\em Israel Journal of Mathematics}, to appear.

\bibitem{fox2012erdHos}
J.~Fox, J.~Pach, B.~Sudakov, and A.~Suk.
\newblock Erd{\H{o}}s--{S}zekeres-type theorems for monotone paths and convex
  bodies.
\newblock {\em Proceedings of the London Mathematical Society},
  105(5):953--982, 2012.

\bibitem{gishboliner2023ramsey}
L.~Gishboliner, Z.~Jin, and B.~Sudakov.
\newblock Ramsey problems for monotone paths in graphs and hypergraphs.
\newblock {\em arXiv preprint arXiv:2308.04357}, 2023.

\bibitem{gyarfas1973ramsey}
A.~Gy{\'a}rf{\'a}s and J.~Lehel.
\newblock A {R}amsey-type problem in directed and bipartite graphs.
\newblock {\em Period. Math. Hungar}, 3(3-4):299--304, 1973.

\bibitem{kHovari1954problem}
T.~K{\H{o}}v{\'a}ri, V.~T.~S{\'o}s, and P.~Turán.
\newblock On a problem of {K.} {Z}arankiewicz.
\newblock In {\em Colloquium Mathematicum}, volume~3, pages 50--57, 1954.

\bibitem{moshkovitz2014ramsey}
G.~Moshkovitz and A.~Shapira.
\newblock Ramsey theory, integer partitions and a new proof of the
  {E}rd{\H{o}}s--{S}zekeres theorem.
\newblock {\em Advances in Mathematics}, 262:1107--1129, 2014.

\bibitem{mubayi2017variants}
D.~Mubayi.
\newblock Variants of the {E}rd{\H{o}}s--{S}zekeres and {E}rd{\H{o}}s--{H}ajnal
  {R}amsey problems.
\newblock {\em European Journal of Combinatorics}, 62:197--205, 2017.

\bibitem{mubayi2017off}
D.~Mubayi and A.~Suk.
\newblock Off-diagonal hypergraph {R}amsey numbers.
\newblock {\em Journal of Combinatorial Theory, Series B}, 125:168--177, 2017.

\bibitem{mubayi2023ramsey}
D.~Mubayi and A.~Suk.
\newblock Ramsey numbers of cliques versus monotone paths.
\newblock {\em arXiv preprint arXiv:2303.16995}, 2023.

\bibitem{spencer1977asymptotic}
J.~Spencer.
\newblock Asymptotic lower bounds for {R}amsey functions.
\newblock {\em Discrete Mathematics}, 20:69--76, 1977.

\end{thebibliography}

\end{document}